\documentclass[10pt]{article}
\usepackage{color,latexsym,amsthm,amsmath,amssymb,path,enumitem,url,verbatim}

\newcommand{\ignore}[1]{}

\makeatletter
\def\th@plain{%
  \thm@notefont{}% same as heading font
  \itshape % body font
}
\def\th@definition{%
  \thm@notefont{}% same as heading font
  \normalfont % body font
}
\makeatother
\newtheorem{theorem}{Theorem}[section]
\newtheorem{lemma}[theorem]{Lemma}
\newtheorem{corollary}[theorem]{Corollary}

\newtheorem{conjecture}[theorem]{Conjecture}

\makeatletter
\newtheorem*{rep@theorem}{\rep@title}
\newcommand{\newreptheorem}[2]{%
\newenvironment{rep#1}[1]{%
 \def\rep@title{#2 \ref{##1}}%
 \begin{rep@theorem}}%
 {\end{rep@theorem}}}
\makeatother
\newreptheorem{theorem}{Theorem}
\usepackage{graphicx,psfrag}

\usepackage{dsfont}

\newcommand{\RR}{\ensuremath{\mathbb R}}

\newcommand{\HH}{\ensuremath{\mathbb H}}

\newcommand{\I}{\mathcal I}
\newcommand{\pts}{\mathcal P}
\newcommand{\lines}{\mathcal L}
\renewcommand{\epsilon}{\varepsilon}

\def \eps {\varepsilon}

\begin{document}
\pagenumbering{arabic}

\title{An improved sum-product bound for quaternions}

\author{
Abdul Basit\thanks{
Department of Mathematics, 255 Hurley Hall, University of Notre Dame, Notre Dame, IN 46556;
{\sl abasit@nd.edu}}
\and
Ben Lund\thanks{
Department of Mathematics, Fine Hall, Princeton University, Princeton NJ 08544; {\sl lund.ben@gmail.com}. 
}
}

%\date{}

\maketitle

\begin{abstract}
We show that there exists an absolute constant $c > 0$, such that, for any finite set $A$ of quaternions, \[ \max\{|A+A|, |AA| \} \gtrsim |A|^{4/3 + c}. \]
This generalizes a sum-product bound for real numbers proved by Konyagin and Shkredov.
\end{abstract}

%%%%%%%%%%%%%%%%%%%%%%%%%%%%%%%%%%%%%%%%%%%%%%%%%%%% INTRODUCTION

\section{Introduction}

 By $X \gg Y$ or $Y \ll X$, we mean that $X \geq cY$, for some absolute constant $c > 0$. The expression $X \approx Y$ means that both $X \gg Y$ and $X \ll Y$ hold. The expression $X \gtrsim Y$ or $Y \lesssim X$ means that $X \gg Y/(\log X)^c$, for some absolute constant $c > 0$. When dependence on a certain parameter needs to be emphasized, we write the parameter in the subscript, e.g., $Y \ll_\epsilon X$ means that the hidden constant depends on $\epsilon$. All logarithms have base 2.

Given finite subsets $A, B$ of a ring, the {\em sum set} and {\em product set} are defined respectively as \[ A + B := \{a + b: a \in A, b \in B \},\] and \[AB := \{ab:  a \in A, b \in B \}. \]

Erd\H{o}s and Szemer\'edi~\cite{es83} conjectured that, for finite sets of integers, one of these must be nearly as large as possible.
\begin{conjecture}[Erd\H{o}s, Szemer\'edi]\label{conj:sumProduct}
	Let $A$ be a finite set of integers. Then, for any $\delta < 1$,
	\begin{equation}\label{eq:sumProduct}
	|A+A| + |AA| \gg_\delta |A|^{1+\delta}. \end{equation}
\end{conjecture}

In their initial work on the problem, Erd\H{o}s and Szemer\'edi showed that Conjecture \ref{conj:sumProduct} holds for some fixed $\delta>0$.
Subsequent works by Nathanson~\cite{nathanson97}, Ford~\cite{ford98}, Elekes~\cite{elekes97}, and Solymosi~\cite{solymosi08} gave increasing values for $\delta$.
Solymosi used a beautiful and simple geometric argument to show that Conjecture~\ref{conj:sumProduct} holds with $\delta \leq 1/3 - \eps$ for any $\eps > 0$ whenever $A$ is a set of real numbers.
This bound stood until 2015, when Konyagin and Shkredov~\cite{ks15} combined Solymosi's geometric insight with Shkredov's work in additive combinatorics to get a slight improvement.
Further incremental progress was made by Konyagin and Shkredov~\cite{ks16}, by Rudnev, Shkredov and Stevens~\cite{rss16}, and by Shakan~\cite{shakan18}. Shakan's result gives the current best bound for $\delta$, showing that Conjecture~\ref{conj:sumProduct} holds with $\delta \leq 1/3 + 5/5277$, whenever $A$ is a finite set of real numbers.

The conjecture has also been studied for other fields and rings.  Konyagin and Rudnev~\cite{kr13} generalized Solymosi's geometric argument to finite sets of complex numbers, showing that Conjecture~\ref{conj:sumProduct} holds with $\delta \leq 1/3 - \eps$ for any $\eps > 0$ whenever $A \subset \mathbb{C}$. For quaternions, Chang~\cite{chang05} proved the bound $\delta \leq 1/54$. This was improved upon by Solymosi and Wong~\cite{sw17}, who showed that, for finite sets of the quaternions, Conjecture~\ref{conj:sumProduct} holds with $\delta \leq 1/3 - \epsilon$, for any $\epsilon > 0$. 
For more detail on the sum-product conjecture and its variants, see the  recent survey of Granville and Solymosi~\cite{gs16}.

Our contribution is to generalize Konyagin and Shkredov's proof to quaternions, hence passing the $\delta=1/3$ barrier.
\begin{theorem}\label{th:main}
	There exists an absolute constant $c > 0$ such that, for any finite set $A$ of quaternions,
	\begin{equation}\label{eq:main}
	|A+A| + |AA| \gg |A|^{4/3 + c}.
	\end{equation}
\end{theorem}
We make no attempt to prove Theorem~\ref{th:main} for the largest possible value of $c$, instead preferring to keep the exposition relatively simple and self contained.

\paragraph{Overview}
Our proof follows the general outline of Konyagin-Shkredov in \cite{ks15} and \cite{ks16}. Since our aim is to keep this paper self contained rather than obtain the best value for $c$, at various points we make weaker estimates than the ones used in these papers.

We split the problem into two cases, depending on the additive energy of $A$ (see Section~\ref{sec:prelims} for the definitions).
In the case that this additive energy is small, we prove an appropriate generalization of Solymosi's argument, in the spirit of Konyagin-Rudnev~\cite{kr13} and Solymosi-Wong~\cite{sw17} (see Section~\ref{sec:large-k}).
In the case that this additive energy is large, we adapt the arguments of Konyagin and Shkredov~\cite{ks16} to work for quaternions (see Section~\ref{sec:small-k}).
This requires us to replace an application of the Szemer\'edi-Trotter theorem by a generalization proved by Solymosi and Tao~\cite{st12}, and to adapt the definitions and arguments of Konyagin and Shrkedov to work when multiplication is not commutative. This is done in Section~\ref{sec:energies}.

%%%%%%%%%%%%%%%%%%%%%%%%%%%%%%%%%%%%%%%%%%%%%%%%%%%%%
%% Section: Preliminaries
%%%%%%%%%%%%%%%%%%%%%%%%%%%%%%%%%%%%%%%%%%%%%%%%%%%%%

\section{Preliminaries}
\label{sec:prelims}

Given finite subsets $A, B$ of the quaternions, $\HH$, the {\em sum set} and {\em product set} are defined respectively as \[ A + B := \{a + b: a \in A, b \in B \},\] and \[AB := \{ab:  a \in A, b \in B \}. \]
We define the {\em negation} of $A$ to be  \[ -A := \{-a : a \in A \},  \] 
and the {\em inverse} of $A$ to be \[ A^{-1} := \{a^{-1} : a \in A,\ a \neq 0 \}. \]
The {\em difference set} is defined to be $A - B$, and the {\em ratio set} is defined as $A/B := AB^{-1} \cap B^{-1}A$. 

Denote by $\delta_{A + B}(x)$ and by $\delta_{AB}(x)$ the number of representations of $x$ of the form $a + b$ and $ab$ with $a \in A$, $b \in B$, respectively.  Let
\[  E^+_k (A, B) := \sum_{x} \delta_{A-B}(x)^k, \quad \mbox{and } \quad E^*_k (A, B) := \sum_{x} \delta_{AB}(x)^k. \]
The {\em additive energy of $A$ and $B$}, denoted by $E^+(A, B)$, is defined to be:
\begin{align*}
E^+(A, B) & := E^+_2(A, B) = \sum_{x} \delta_{A - B}(x)^2 = |\{(a, b, c, d) \in (A \times B)^2: a - b = c - d\}| \\
& =  |\{(a, b, c, d) \in (A \times B)^2: a - c = b - d\}| = \sum_{x} \delta_{A - A}(x)\delta_{B-B}(x)\\
& = |\{(a, b, c, d) \in (A \times B)^2: a + d = c + b\}| = \sum_{x} \delta_{A + B}(x)^2.
\end{align*}
Similarly, the {\em multiplicative energy of $A$ and $B$}, denoted by $E^*(A, B)$, is defined to be 
\begin{align*}
E^*(A, B) := E^*_2(A, B)
& = \sum_{x} \delta_{AB}(x)^2 = |\{(a, b, c, d) \in (A \times B)^2: ab = cd\}|.\\
& = |\{(a, b, c, d) \in (A \times B)^2: c^{-1}a = db^{-1} \}| = \sum_{x} \delta_{A^{-1}A}(x) \delta_{BB^{-1}}(x).
\end{align*}
When $A = B$, we write $E^+_k(A)$, $E^+(A)$, $E^*_k(A)$, and $E^*(A)$ to simplify notation. We note that, since multiplication of quaternions is not commutative, some care is necessary when dealing with product sets, ratio sets, and the multiplicative energy.

The Cauchy-Schwarz inequality implies the following lower bounds on the additive and multiplicative energies:
\begin{equation}
\label{eq:additivecs}
E^+(A, B) \geq \frac{|A|^2|B|^2}{|A + B|}, \quad \mbox{ and }\quad E^+(A, B) \geq \frac{|A|^2|B|^2}{|A - B|}.
\end{equation}
\begin{equation}
\label{eq:multiplicativecs}
E^*(A, B) \geq \frac{|A|^2|B|^2}{|AB|}.
\end{equation}

%%%%%%%%%%%%%%%%%%%%%%%%%%%%%%%%%%%%%%%%%%%%%%%%%%%%%
%% Section: Incidences
%%%%%%%%%%%%%%%%%%%%%%%%%%%%%%%%%%%%%%%%%%%%%%%%%%%%%

\subsection{Results from incidence geometry}

In this section, we set up some notation and results about {\em quaternionic lines}. In Section~\ref{sec:energies}, these results will be used to connect various energies with the sizes of the sum and product sets.

Define a {\em left line} to be any set in $\HH^2$ of the form $\{(a,b)+t(c,d) : t \in \HH\}$, for some $a,b,c,d \in \HH$ with $(c, d) \neq (0, 0)$. Similarly, define {\em right}, {\em lr-mixed} and {\em rl-mixed lines} to be sets of the form $\{(a,b)+(c,d)t : t \in \HH\}$, $\{(a,b)+(tc,dt) : t \in \HH\}$, and $\{(a,b)+(ct,td) : t \in \HH\}$ respectively. We say two lines are of the {\em same type} if both are sets of the same form. It is straightforward to check that any two distinct quaternionic lines {\em of the same type} intersect in at most one point. Throughout this writeup, a set of quaternionic lines is restricted to have all lines of the same type.

Given a point set $\pts$ and a set of quaternionic lines $\lines$, both in $\HH^2$, we say that an incidence is a pair $(p,l) \in \pts \times \lines$ with the point $p$ lying on the line $l$. We denote by $\I(\pts, \lines)$ the set of incidences in $\pts \times \lines$. The following theorem is a special case of a result of Solymosi and Tao~\cite{st12}:

%%%%%%%%%%%%%%%%%%%%%%%%%%%%%%%%%%%%%%%%%%%%%%%%%%%%%
%%%%%%%%%%%%%%%%%%%%%%%%%%%%%%%%%%%%%%%%%%%%%%%%%%%%%

\begin{theorem}
\label{th:solymosi-tao}
Let $\epsilon > 0$. Then there exists a constant $A = A_\epsilon > 0$ such that \[ |\I(\pts, \lines)| \leq A|\pts|^{2/3 + \epsilon}|\lines|^{2/3} + \frac{3}{2}|\pts| + \frac{3}{2}|\lines| \] whenever $\pts$ is a finite set of points in $\HH^2$, and $\lines$ is a finite set of quaternionic lines in $\HH^2$.
\end{theorem}
As a consequence, we have the following upper bound on the number of {\em $k$-rich lines} determined by a finite set in $\HH^2$:
\begin{corollary}
\label{co:solymosi-tao}
Let $\epsilon > 0$. Then for any set $\pts$ of points in $\HH^2$ and, for any integer $k \geq 2$, the number of quaternionic lines containing at least $k$ points of $\pts$, denoted by $n_k$, satisfies  \[ n_k \ll_\epsilon \frac{|\pts|^{2+\epsilon}}{k^3} + \frac{|\pts|}{k} . \]
\end{corollary}

%%%%%%%%%%%%%%%%%%%%%%%%%%%%%%%%%%%%%%%%%%%%%%%%%%%%%
%%%%%%%%%%%%%%%%%%%%%%%%%%%%%%%%%%%%%%%%%%%%%%%%%%%%%

Next, we use Corollary~\ref{co:solymosi-tao} to obtain the following upper bound on the number of collinear triples:
\begin{lemma}
\label{le:collinear-triples}
Let $\epsilon > 0$. For a finite set $A$ of quaternions, let $T$ be the number of collinear triples in $A \times A \subset \HH^2$ (where the points in a triple are not necessarily distinct). Then \[ T \ll_\epsilon |A|^{4+\epsilon}. \]
\end{lemma}
\begin{proof}
Consider a partitioning of the set of lines into dyadic blocks, where the $i^{th}$ block consists of lines containing between $2^i$ and $2^{i+1}$ points, with $1 \leq i \leq \lceil \log |A| \rceil$. Let $n_i$ be the number of lines in the $i^{th}$ set. 
Then the number of collinear triples is
\begin{equation}
\label{eq:triples}
T \leq \sum_{i = 1}^{\lceil \log |A| \rceil} \left(2^{i+1}\right)^3 n_{i}.
\end{equation}
This allows us to use Corollary~\ref{co:solymosi-tao}, with $|\pts| = |A|^2$ and $\epsilon' = \epsilon/2$, to bound $n_i$, giving
\begin{align*}
T & \leq \sum_{i = 1}^{\lceil \log |A| \rceil} \left(2^{i+1}\right)^3 n_{i}  \ll_{\epsilon'} \sum_{i = 1}^{\lceil \log |A| \rceil} \left(2^{i+1}\right)^3 \left(\frac{|A|^{4+{\epsilon'}}}{(2^{i})^3} + \frac{|A|^2}{2^{i}} \right)\\
& \ll \sum_{i = 1}^{\lceil \log |A| \rceil} \left(|A|^{4+{\epsilon'}} + |A|^2 2^{2i} \right)  \ll |A|^{4+2\epsilon'}.
\end{align*}
\end{proof}

%%%%%%%%%%%%%%%%%%%%%%%%%%%%%%%%%%%%%%%%%%%%%%%%%%%%%
%% Section: Energies
%%%%%%%%%%%%%%%%%%%%%%%%%%%%%%%%%%%%%%%%%%%%%%%%%%%%%

\section{Bounding energies}
\label{sec:energies}

In this section we collect various bounds on the sizes of sum and product in terms of energies. All the results and proofs presented in this section have appeared in various papers. We simply present the proofs and adapt them to our setting.

%%%%%%%%%%%%%%%%%%%%%%%%%%%%%%%%%%%%%%%%%%%%%%%%%%%%%
%%%%%%%%%%%%%%%%%%%%%%%%%%%%%%%%%%%%%%%%%%%%%%%%%%%%%

First, we give some basic estimates on additive energies using the Cauchy-Schwarz inequality. The following appears as Lemma 2.4 and 2.5 in~\cite{li11}.
\begin{lemma}
\label{le:e1.5e3}
Let $A, B$ be finite sets of quaternions. Then
\[ |B|^2E^+_{1.5}(A)^2 \leq E_{3}^+(A)^{2/3}  E_{3}^+(B)^{1/3} E^+(A, A+B). \]
\end{lemma}
\begin{proof}
Let $A_x = A \cap (A + x)$, and note that $|A_x| = \delta_{A-A}(x)$. Equation~\eqref{eq:additivecs} implies that
\begin{equation*}
|A_x|^{1.5}|B| \leq E^+(A_x, B)^{1/2}|A_x + B|^{1/2} |A_x|^{1/2}.
\end{equation*}
This gives 
\begin{align}
\label{eq:b1}
\nonumber E^+_{1.5}(A)^2|B|^2 & = \left(\sum_{x \in A-A} \left(|A_x|^{1.5}|B|\right)\right)^2\\
\nonumber &\leq \left( \sum_{x \in A-A} \left(E^+(A_x, B)^{1/2}|A_x + B|^{1/2}|A_x|^{1/2} \right)\right)^2\\
\nonumber &\leq \sum_{x \in A-A} E^+(A_x, B) \sum_{x \in A-A} |A_x + B| |A_x|\\
\nonumber &\leq \sum_{x \in A-A} E^+(A_x, B) \sum_{x \in A-A} |(A + B)_x| |A_x|\\
&= \left(\sum_{x \in A-A} E^+(A_x, B)\right) E^+(A, A+B).
\end{align}
where the second inequality follows from the Cauchy-Schwarz inequality, and the third inequality is implied by the set inclusion $A_x + B \subseteq (A + B)_x$.

Next, note that
\begin{align}
\label{eq:b2}
\nonumber \sum_{x} E^+(A_x, B) &= \sum_{x} \sum_{y} \delta_{A_x - A_x}(y) \delta_{B-B}(y)
= \sum_{x} \sum_{y} \delta_{A_y - A_y}(x) \delta_{B-B}(y)\\
\nonumber &= \sum_{y} \sum_{x} \delta_{A_y - A_y}(x) \delta_{B-B}(y)
= \sum_{y}  \delta_{A - A}(y)^2 \delta_{B-B}(y)\\
&\leq \left( \sum_{y}  \delta_{A - A}(y)^3\right)^{2/3} \left(\sum_{y} \delta_{B-B}(y)^3 \right)^{1/3}
= E_3^+(A)^{2/3}E_3^+(B)^{1/3},
\end{align}
where the second to last step is a standard application of H\"older's inequality.

The statement of the lemma now follows by combining~\eqref{eq:b1}~and~\eqref{eq:b2}.

\end{proof}

%%%%%%%%%%%%%%%%%%%%%%%%%%%%%%%%%%%%%%%%%%%%%%%%%%%%%
%%%%%%%%%%%%%%%%%%%%%%%%%%%%%%%%%%%%%%%%%%%%%%%%%%%%%

The following lemma establishes a connection between the additive energy and the size of the product set, and appears as Theorem~9 in~\cite{ks15}. The technique used here was introduced by Elekes and Ruza in~\cite{er03}. 
\begin{lemma}
\label{le:additive}
Let $\epsilon > 0$, and $A$ be a finite set of quaternions, with $0 \notin A$. Then \[ E^+(A)^4 \ll_{\epsilon} |AA| |A|^{10+\epsilon}. \]
\end{lemma}
\begin{proof}
In order to bound $E^+(A)$, we may restrict our attention to elements that have many realizations in $A+A$. Let $F$ be the set of such elements, i.e., \[ F = \left\{ x \in A+A: \delta_{A+A}(x) > \frac{E^+(A)}{2|A|^2} \right\}. \]
Now \[  \sum_{x \notin F} \delta_{A+A}(x)^2 \leq \frac{E^+(A)}{2|A|^2} \sum_{x \notin F} \delta_{A+A}(x) \leq \frac{E^+(A)}{2}. \]
In particular, we have \[  E^+(A) \leq 2 \sum_{x \in F} \delta_{A+A}(x)^2. \]
Let $\mathcal{P} = (A \cup F) \times (A \cup F)$. We will double count the number of collinear triples in $\mathcal{P}$ (where points in a triple are not necessarily distinct), denoted by~$T$.

Denote by $\delta_F$ the quantity $\sum_{x \in F} \delta_{A+A}(x)$. Since, for all $x$, $\delta_{A+A}(x) \leq |A|$, we have
\[  \delta_F = \sum_{x \in F} \delta_{A+A}(x) \geq \sum_{x \in F}  \frac{\delta_{A+A}(x)^2}{|A|} \geq \frac{E^+(A)}{2|A|}. \] 
This gives
\[ |F| + |A| \leq \sum_{x \in F} \left( \delta_{A+A}(x) \cdot \frac{2|A|^2}{E^+(A)}\right) + |A|\left( \delta_F\cdot \frac{2|A|}{E^+(A)} \right)
= \frac{4|A|^2\delta_F}{E^+(A)}. \] 
Combining the above with Lemma~\ref{le:collinear-triples} (with $\epsilon' = \epsilon/2$) implies
\begin{equation}
\label{eq:triplesupperbound}
T \ll_{\epsilon'} |A \cup F|^{4+{\epsilon'}} \ll \left( \frac{|A|^2 \delta_F}{E^+(A)} \right)^{4+{\epsilon'}}.
\end{equation}

We now obtain a lower bound on $T$. For each $a \in A$, let $F_a = \{b \in A: a + b \in F\}$. Fix $(a, b) \in A^2$, and consider a quadruple of the form $(c, d, e, f) \in (F_a \times F_b)^2$ with $cd = ef$, or equivalently $df^{-1} = c^{-1}e$. 
Note that the line $\{(a, b) + (ct, tf)  : t \in \HH \}$ contains the points 
\[ (a, b), (a + c, b + f), (a + e, b + d) \in \pts, \]
with $t = 0,\, t = 1$, and $t = df^{-1} = c^{-1}e $ respectively. 
Since the quadruples $(c, d, e, f)$ and $(e, f, c, d)$ give the same collinear triple, by Equation~\eqref{eq:multiplicativecs}, each pair $(a, b) \in A^2$ gives at least \[ \frac{1}{2}\left( \frac{|F_a|^2|F_b|^2}{|F_a F_b|} \right) \gg \frac{|F_a|^2|F_b|^2}{|AA|} \] 
 distinct collinar triples. It follows that the number of collinear triples in $\pts$ is at least
\begin{align}
\label{eq:tripleslowerbound}
\nonumber T &\gg \sum_{a, b \in A} \frac{|F_a|^2|F_b|^2}{|AA|}  = \frac{1}{|AA|} \left(\sum_{a \in A } |F_a|^2\right)^2\\
&\geq \frac{1}{|AA|} \left(\frac{1}{|A|} \left(\sum_{a \in A} |F_a|\right)^2 \right)^2 = \frac{\delta_F^4}{|AA||A|^2}.
\end{align}
Combining Equations~\eqref{eq:triplesupperbound}~and~\eqref{eq:tripleslowerbound} gives:
\[ E^+(A)^{4 + {\epsilon'}} \ll_{\epsilon'} |A|^{10 + 2{\epsilon'}} \delta_F^{{\epsilon'}} |AA|. \]
Finally, since $E^+(A) \geq |A|^2$ and $\delta_F = \sum_{x \in F} \delta_{A+A}(x) \leq |A|^2$,
we get \[ E^+(A)^{4} \ll_\epsilon |AA||A|^{10 + \epsilon}. \]
\end{proof}

%%%%%%%%%%%%%%%%%%%%%%%%%%%%%%%%%%%%%%%%%%%%%%%%%%%%%
%%% d_* definition
%%%%%%%%%%%%%%%%%%%%%%%%%%%%%%%%%%%%%%%%%%%%%%%%%%%%%

An important tool in connecting incidence results and sum-product estimates is the quantity $d_*(A)$ defined below, which captures the multiplicative structure of $A$. We define \[ d_*(A) := \min_{t > 0} \min_{\varnothing \neq Q, R \subset \HH \setminus \{0\}} \frac{|Q|^2|R|^2}{|A|t^3}, \] where the second minimum is taken over all sets $Q$ and $R$ such that $|Q| \geq \max\{|A|, |R|\}$, and, for every $a \in A$, the bound $|Q \cap R a| \geq t$ holds.

\vspace{-0.5em}\paragraph{Comment:} Note that the definition of $d_*(A)$ is different from that in \cite{ks16}, where it was only required that $\max\{|Q|, |R|\} \geq |A|$. Non-commutativity of multiplication for the quaternions means that we require the stronger condition that $|Q| \geq \max\{|A|, |R|\}$. When $A$ is a subset of the real numbers (or the complex numbers), we can assume, without loss of generality, that $|Q| \geq |R|$. This change in definition does not affect any of the proofs later.\\

The sets $Q=A$ and $R=\{1\}$ show that $d_*(A) \leq |A|$. Roughly speaking, the closer $d_*(A)$ is to $|A|$, the less multiplicative structure $A$ has.  The following is a key lemma concerning $d_*(A)$ and will be used to connect additive energies with the product set. The version for real numbers appeared as Lemma~13 in~\cite{ks16}.

%%%%%%%%%%%%%%%%%%%%%%%%%%%%%%%%%%%%%%%%%%%%%%%%%%%%%
%%%%%%%%%%%%%%%%%%%%%%%%%%%%%%%%%%%%%%%%%%%%%%%%%%%%%

\begin{lemma}
\label{le:stsets}
Let $\epsilon > 0$, and $A, B$ be finite sets of quaternions, and $\tau \geq 1$ an integer. Then \[ |\{x : \delta_{A-B}(x) \geq \tau\}| \ll_\epsilon \frac{|A||B|^{2 + \epsilon}}{\tau^3} d_*(A)^{1+\epsilon}. \]
\end{lemma}
\begin{proof}
Note that we may assume $A$ and $B$ are large (with respect to the hidden constant), since otherwise the statement is trivially true. Let $Q, R \subset \HH\setminus\{0\}$ with $|Q| \geq \max\{|A|, |R|\}$, and $t > 0$ an integer such that, for all $a \in A$, $|Q \cap R a| \geq t$. For any integer $\tau \geq 1$, let
\[D_\tau = \{x \in A - B: \delta_{A - B}(x) \geq \tau \}. \]
To prove the lemma, it suffices to show that
\begin{equation}
\label{eq:required}
|D_\tau| \ll_\epsilon \frac{|Q|^{2+\epsilon}|R|^{2+\epsilon}|B|^{2+\epsilon}}{t^3\tau^3}.
\end{equation}
Let $\sigma = |\{ (a, b, x) : a - b = x, \, a \in A, b \in B, x \in D_\tau \}|$, and note that, by definition,
\begin{equation}
\label{eq:sigmalb}
\sigma = \sum_{x \in D_\tau} \delta_{A-B}(x) \geq \tau|D_\tau|.
\end{equation}
On the other hand, for each $a \in A$, $|Q \cap Ra| \geq t$ implies that $\delta_{R^{-1}Q}(a) \geq t$. 
This gives
\begin{equation}
\label{eq:sigmaub}
\sigma \leq |\{(r, q, b, x): r^{-1}q - b = x,\, r \in R, q \in Q, b \in B, x \in D_\tau\}|\cdot t^{-1}.
\end{equation}
By ignoring the condition $x \in D_\tau$, Equation~\eqref{eq:sigmaub} immediately implies the bound
\begin{equation}
\label{eq:sigmaub-easy}
\sigma \leq \frac{|Q||R||B|}{t}.
\end{equation}
When $t^2\tau^2 \ll |Q||R||B|$, Equations~\eqref{eq:sigmalb}~and~\eqref{eq:sigmaub-easy} imply
\[ |D_\tau| \leq \frac{|Q||R||B|}{t\tau} \ll \frac{|Q|^2|R|^2|B|^2}{t^3\tau^3}, \]
allowing us to restrict our attention to the case when $t^2\tau^2 \gg |Q||R||B|$.

In this scenario, we reformulate Equation~\eqref{eq:sigmaub} as an incidence problem. For $r \in R$ and $d \in D_\tau$, let $l_{r,d} = \{ (x, y) : r^{-1}y - x = d \}$ be a quaternionic line. Consider the family $\lines = \{ l_{r,d}:  r \in R, d \in D_\tau \}$ of $|R||D_\tau|$ lines. Let $\pts$ be the point set $B \times Q$. Now, by Theorem~\ref{th:solymosi-tao},  for $\epsilon' = \epsilon/3$, we have
\begin{align}
\label{eq:sigmaub1}
\nonumber \sigma & \leq |\I(\pts, \lines)| \cdot t^{-1}\\
& \ll \left(A_{\epsilon'} |\pts|^{2/3 + \epsilon'}|\lines|^{2/3} + |\pts| + |\lines|\right) \cdot t^{-1}.
\end{align}
If the first term in Equation~\eqref{eq:sigmaub1} dominates, we have
\begin{equation*}
t \tau |D_\tau| \ll A_{\epsilon'} |\pts|^{2/3 + \epsilon'}|\lines|^{2/3} = A_{\epsilon'}(|Q||B|)^{2/3 + \epsilon'}(|R||D_\tau|)^{2/3}.
\end{equation*}
Rearranging gives
\begin{align*}
|D_\tau| & \ll_\epsilon \frac{|Q|^{2 + \epsilon}|B|^{2 + \epsilon}|R|^2}{t^3\tau^3}.
\end{align*}

Now suppose that the bound \eqref{eq:required} does not hold. If the second term dominates, then
\begin{equation}
\label{eq:stcontradiction1}
\frac{|Q|^2|R|^2|B|^2}{t^2\tau^2} \ll t \tau |D_\tau| \ll |\pts| = |Q||B|.
\end{equation}
Note that we must have $t \leq \min\{|Q|, |R|\} = |R|$, $\tau \leq |B|$, and $\tau \leq |A| \leq |Q|$. Together with Equation~\eqref{eq:stcontradiction1}, this gives a contradiction.

Finally, if the third term dominates, then
\[ t\tau |D_\tau| \ll |\lines| = |R||D_\tau|. \]
Combined with the assumption that $t^2\tau^2 \geq |Q||R||B|$, this implies
\[ |Q||B||R| \ll |R|^2. \]
But this is a contradiction, since $|Q| \geq |R|$, and $B$ is large enough.
\end{proof}

%%%%%%%%%%%%%%%%%%%%%%%%%%%%%%%%%%%%%%%%%%%%%%%%%%%%%
%%%%%%%%%%%%%%%%%%%%%%%%%%%%%%%%%%%%%%%%%%%%%%%%%%%%%
We now give some consequences of Lemma~\ref{le:stsets}. The proof of all three corollaries follows the same basic outline.
\begin{corollary}
\label{co:e1.5}
Let $\epsilon > 0$, and $A$ be a finite set of quaternions. Then
\[ E^+(A) \ll_{\epsilon} |A|^{1 + \epsilon} d_*(A)^{1/3 + \epsilon} E_{1.5}^+(A)^{2/3}. \]
\begin{proof}
Let $\Delta$ be a parameter to be specified later, and recall that
\[ E^+(A) = \sum_{x} \delta_{A-A}(x)^2 = \sum_{x\,:\,\delta_{A-A}(x) < \Delta} \delta_{A-A}(x)^2  + \sum_{x\,:\,\delta_{A-A}(x) \geq \Delta} \delta_{A-A}(x)^2. \]
We first consider sums with fewer than $\Delta$ realizations.
\begin{align}
\label{eq:a1}
\nonumber \sum_{x\,:\,\delta_{A-A}(x) < \Delta} \delta_{A-A}(x)^2  &\leq \max_{x\,:\,\delta_{A-A}(x) < \Delta}\{\delta_{A-A}(x)^{1/2}\}\sum_{x\,:\,\delta_{A-A}(x) < \Delta} \delta_{A-A}(x)^{1.5}\\ 
&\leq \Delta^{1/2}E^+_{1.5}(A).
\end{align}
To bound the contribution of sums with more than $\Delta$ realizations, we use a dyadic decomposition along with Lemma~\ref{le:stsets} (with $B = A$ and $\epsilon' = 3\epsilon/2$).
\begin{align}
\label{eq:a2}
\nonumber \sum_{x\,:\,\delta_{A-A}(x) \geq \Delta} \delta_{A-A}(x)^2 & = \sum_{j = 1}^{\lceil \log |A| \rceil + 1} \sum_{x\,:\, \Delta2^{j-1} \leq \delta_{A-A}(x) < \Delta 2^{j}} \delta_{A-A}(x)^2\\
\nonumber& \ll_{\epsilon'} \sum_{j = 1}^{\lceil \log |A| \rceil + 1} \frac{|A|^{3+\epsilon'}d_*(A)^{1 + \epsilon'}}{(\Delta2^{j-1})^3} (\Delta 2^{j})^2\\
& \ll \frac{|A|^{3+2\epsilon'}d_*(A)^{1 + 2\epsilon'}}{\Delta}.
\end{align}
Combining~\eqref{eq:a1},~\eqref{eq:a2}, and setting $\Delta = \left(\frac{|A|^{3+2\epsilon'}d_*(A)^{1 + 2\epsilon'}}{E^+_{1.5}(A)}\right)^{2/3}$ finishes the proof.
\end{proof}
\end{corollary}

%%%%%%%%%%%%%%%%%%%%%%%%%%%%%%%%%%%%%%%%%%%%%%%%%%%%%
%%%%%%%%%%%%%%%%%%%%%%%%%%%%%%%%%%%%%%%%%%%%%%%%%%%%%
\begin{corollary}
\label{co:eds}
Let $\epsilon > 0$, and $A, B$ be finite sets of quaternions. Then
\[  E^+(A, B) \ll_\epsilon |A||B|^{3/2 + \epsilon}d_*(A)^{1/2 + \epsilon} . \] 
\end{corollary}
\begin{proof}
Let $\Delta$ be a parameter to be specified later. Then
\begin{align*}
E^+(A, B) & = \sum_{x} \delta_{A-B}(x)^2\\
& = \sum_{x\,:\,\delta_{A-B}(x) < \Delta} \delta_{A-B}(x)^2  + \sum_{j = 1}^{\lceil \log |B| \rceil + 1} \sum_{x\,:\, \Delta2^{j-1} \leq \delta_{A-B}(x) < \Delta 2^{j}} \delta_{A-B}(x)^2\\
& \ll_{\epsilon} \Delta |A||B| + \sum_{j = 1}^{\lceil \log |B| \rceil + 1} \frac{|A||B|^{2 + \epsilon}d_*(A)^{1 + \epsilon}}{(\Delta2^{j-1})^3}(\Delta2^j)^2\\
&\ll \Delta |A||B| + \frac{|A||B|^{2 + 2\epsilon}d_*(A)^{1 + 2\epsilon}}{\Delta}.
\end{align*}
Setting $\Delta = \left(|B|^{1 + 2\epsilon}d_*(A)^{1 + 2\epsilon}\right)^{1/2}$ completes the proof.
\end{proof}

%%%%%%%%%%%%%%%%%%%%%%%%%%%%%%%%%%%%%%%%%%%%%%%%%%%%%
%%%%%%%%%%%%%%%%%%%%%%%%%%%%%%%%%%%%%%%%%%%%%%%%%%%%%
\begin{corollary}
\label{co:e3ds}
Let $\epsilon > 0$, and $A$ be a finite set of quaternions. Then
\[  E^+_3(A) \ll_\epsilon |A|^{3 + \epsilon} d_*(A)^{1 + \epsilon}. \]
\end{corollary}
\begin{proof}
\begin{align*}
E_3^+(A) & = \sum_{x} \delta_{A-A}(x)^3  = \sum_{j = 1}^{\lceil \log |A| \rceil + 1} \sum_{x\,:\, 2^{j-1}
\leq \delta_{A-A}(x) <  2^{j}} \delta_{A-A}(x)^3\\
& \ll_{\epsilon'} \sum_{j = 1}^{\lceil \log |A| \rceil + 1} \frac{|A|^{3 + \epsilon'}d_*(A)^{1 + \epsilon'}}{(2^{j-1})^3}(2^j)^3
\ll |A|^{3 + 2\epsilon'}d_*(A)^{1 + 2\epsilon'}.
\end{align*}
\end{proof}

%%%%%%%%%%%%%%%%%%%%%%%%%%%%%%%%%%%%%%%%%%%%%%%%%%%%%
%%%%%%%%%%%%%%%%%%%%%%%%%%%%%%%%%%%%%%%%%%%%%%%%%%%%%
We now give the main result of this section. The following theorem will be used in Section~\ref{sec:small-k} when the multiplicative energy is small.
\begin{theorem}
\label{th:smallk}
Let $\epsilon > 0$, and $A$ be a finite set of quaternions. Then \[ |A+A| \gg_{\epsilon} \frac{|A|^{14/9 - \epsilon}}{d_*(A)^{5/9}}. \]
\end{theorem}
\begin{proof}
The proof combines Inequality~\eqref{eq:additivecs}, Corollary~\ref{co:e1.5}, Lemma~\ref{le:e1.5e3}, Corollary~\ref{co:eds}, and Corollary~\ref{co:e3ds} (in the specified order). This gives, for $\epsilon' = 9\epsilon/22$,
\begin{align*}
\frac{|A|^{12}}{|A+A|^3} &\leq E^+(A)^3 \\
&\ll_{\epsilon'} |A|^{3 + 3\epsilon'}d_*(A)^{1 + 3\epsilon'}E^+_{1.5}(A)^2\\
&\leq |A|^{1 + 3\epsilon'}d_*(A)^{1 + 3\epsilon'}E_3^+(A)E^+(A, A+A)\\
&\ll_{\epsilon'} |A|^{2 + 3\epsilon'}d_*(A)^{3/2 + 4\epsilon'}E_3^+(A)|A+A|^{3/2 + \epsilon'}\\
&\ll_{\epsilon'} |A|^{5 + 4\epsilon'}d_*(A)^{5/2 + 5\epsilon'}|A+A|^{3/2 + \epsilon'}.
\end{align*}
This implies \[ |A+A|^{9/2 + \epsilon'} \gg_{\epsilon'} \frac{|A|^{7 - 4\epsilon'}}{d_*(A)^{5/2 + 5\epsilon'}}. \]
The theorem now follows by noting that $|A+A| \leq |A|^2$ and $d_*(A) \leq |A|$.
\end{proof}

%%%%%%%%%%%%%%%%%%%%%%%%%%%%%%%%%%%%%%%%%%%%%%%%%%%%%
%% Section: Main Theorem
%%%%%%%%%%%%%%%%%%%%%%%%%%%%%%%%%%%%%%%%%%%%%%%%%%%%%

\section{Main Theorem}
\label{sec:proof}

We are finally ready to prove Theorem~\ref{th:main}. The section is organized as follows:
In Section~\ref{sec:initial}, we use standard arguments to find a large subset of $A$ with properties that will be convenient to work with. Section~\ref{sec:small-k} will deal with the case when the additive energy of this subset is large, and Section~\ref{sec:large-k} will deal with the case when the additive energy is small.

\subsection{Initial setup and pigeonholing}
\label{sec:initial}

Let $A$ be a finite set of quaternions. Since we are not concerned about constants, it suffices to prove~\eqref{eq:main} for a subset $A' \subseteq A$ containing a positive proportion of elements of $A$. We view a quaternion $a = w + xi + yj + zk$ as a vector $(w, x, y, z) \in \RR^4$, and say that $w$ and $(x, y, z)$ are, respectively, the {\em real} and {\em imaginary} parts of $a$. By a simple pigeonholing argument, there exists a set $A'$ of size at least $|A|/16$ such that all elements of $A'$ lie in the same orthant of $\RR^4$. We assume that $0 \notin A'$, and that each element of $A'$ has non-negative real part. The latter is without loss of generality since multiplication by reals is commutative and, hence, doesn't affect the sizes of the sum and product sets. To simplify notation, we identify $A'$ with $A$, and assume that $A$ satisfies the properties that we need.

To estimate the multiplicative energy, we will consider a subset of $A/A$, which we build incrementally. Let $R_0 = \{\lambda \in A/A: \delta_{A^{-1}A}(\lambda) \geq \delta_{AA^{-1}}(\lambda)\}$, and  assume, without loss of generality, that
\[ \sum_{\lambda \in R_0} \delta_{A^{-1}A}(\lambda)\delta_{AA^{-1}}(\lambda) \geq \sum_{\lambda \notin R_0} \delta_{A^{-1}A}(\lambda)\delta_{AA^{-1}}(\lambda). \]
This gives \[ E^*(A) = \sum_{\lambda \in A/A} \delta_{A^{-1}A}(\lambda)\delta_{AA^{-1}}(\lambda) \leq 2  \sum_{\lambda \in R_0} \delta_{A^{-1}A}(\lambda)\delta_{AA^{-1}}(\lambda) \leq 2 \sum_{\lambda \in R_0} \delta_{A^{-1}A}(\lambda)^2. \]
Next, we restrict our attention to $\lambda \in R_0$ with $\delta_{A^{-1}A}(\lambda) \geq E^*(A)/4|A|^2$. Let $R_1 = \{\lambda \in R_0: \delta_{A^{-1}A}(\lambda) \geq E^*(A)/4|A|^2\}$. This is possible because ratios in~$R_0 \setminus R_1$ can not account for too much the multiplicative energy, i.e.,
\[ \sum_{\lambda \in R_0 \setminus R_1} \delta_{A^{-1}A}(\lambda)^2 < \frac{E^*(A)}{4|A|^2} \sum_{\lambda \in R_0\setminus R_1} \delta_{A^{-1}A}(\lambda) \leq \frac{E^*(A)}{4} \leq \frac{1}{2}\sum_{\lambda \in R_0} \delta_{A^{-1}A}(\lambda)^2. \]
This implies that ratios in $R_1$ give the following bound on the multiplicative energy: \begin{equation}\label{eq:multiplicative-1} E^*(A) \ll \sum_{\lambda \in R_1} \delta_{A^{-1}A}(\lambda)^2. \end{equation}
For a positive integer $\tau$, let $S'_\tau$ be the set $ S'_\tau = \{\lambda \in R_1: \tau \leq \delta_{A^{-1}A}(\lambda) < 2\tau \}. $
Since $1 \leq  \delta_{A^{-1}A}(\lambda) \leq |A|$, a dyadic decomposition of the summation in~\eqref{eq:multiplicative-1} gives
\[  \sum_{\lambda \in R_1} \delta_{A^{-1}A}(\lambda)^2 = \sum_{i = 0}^{\lceil \log |A| \rceil} \sum_{\lambda \in S'_{2^i}} \delta_{A^{-1}A}(\lambda)^2. \]
It follows, by the pigeonhole principle, that there exists a $\tau \geq E^*(A)/4|A|^2$ such that elements of $S'_\tau$ contribute at least $1/\lceil \log |A|\rceil$ to the sum. This implies \[ E^*(A) \lesssim \sum_{\lambda \in S'_\tau} \delta_{A^{-1}A}(\lambda)^2 \ll |S'_\tau|\tau^2. \]
Finally, for each $\lambda \in S'_\tau$, consider the set $A_\lambda := A \cap A \lambda$. Then $|A_\lambda| = \delta_{A^{-1}A}(\lambda)$, which implies that $\tau^2 \leq E^+(A_\lambda) \leq 8\tau^3$.
A dyadic decomposition of this interval, along with the pigeonhole principle, implies that there exists $K \geq 1$ such that the set $S_\tau = \{ \lambda \in S'_\tau: \tau^3/2K \leq E^+(A_\lambda) < \tau^3/K\} $ has cardinality $|S_\tau| \gtrsim  |S'_\tau|$. 

To sum up, for an integer
\begin{equation}
\label{eq:taubound}
\tau \gg E^*(A)/|A|^2,
\end{equation}
we have found a set $S_\tau \subseteq A/A$ such that
\begin{equation}
\label{eq:energybound}
E^*(A) \lesssim |S_\tau|\tau^2,
\end{equation}
and, for every $\lambda \in S_\tau$, we have \[ \delta_{A^{-1}A}(\lambda) \geq \delta_{AA^{-1}}(\lambda),\,\,  |A_\lambda| = \delta_{A^{-1}A}(\lambda) \approx \tau, \mbox{ and } E^+(A_\lambda) \approx \tau^3/K \mbox{ for some } K \geq 1. \]

%%%%%%%%%%%%%%%%%%%%%%%%%%%%%%%%%%%%%%%%%%%%%%%%%%%%%
%% Section: Small k
%%%%%%%%%%%%%%%%%%%%%%%%%%%%%%%%%%%%%%%%%%%%%%%%%%%%%

\subsection{The case when $K$ is small} 
\label{sec:small-k}

We will show that there exists a {\em large} set $A' \subseteq A$ such that $d_*(A')$ is {\em small}. The result will then follow by using Theorem~\ref{th:smallk}. 

Recall that for $\lambda \in A/A$, $A_\lambda = A \cap A \lambda \subseteq A$. We will require the {\em Katz-Koester} inclusion, i.e., for any $\lambda \in A/A$,  $A_\lambda A_\lambda \subseteq (AA) \cap (AA) \lambda$. Clearly $A_\lambda A_\lambda \subseteq (AA)$, so it suffices to show $A_\lambda A_\lambda \subseteq (AA) \lambda $. To see this, consider $a = b \cdot c \in A_\lambda A_\lambda$. Since $c \in A_\lambda$, there exists $c' \in A$ such that $c =  c' \lambda$, that is, $c' = c \lambda^{-1} \in A$. Now we may write $a =  b \cdot c = b \cdot c \cdot \lambda^{-1} \cdot \lambda = b \cdot c' \lambda \in (AA)\lambda$.

Note that \[ \sum_{a \in A} |A \cap  a S_\tau| = \sum_{\lambda \in S_\tau} |A \cap A \lambda| \geq |S_\tau|\tau. \]
Then, by the pigeonhole principle, there exists $a \in A$ such that $A' := A \cap a S_\tau $ has size at least $|S_\tau|\tau/|A|$. 
Observe that Lemma~\ref{le:additive} implies that, for all $\lambda \in S_\tau$,
\[ |A_\lambda A_\lambda| \gg_{\epsilon'} \frac{E^+(A_\lambda)^4}{|A_\lambda|^{10 + \epsilon'}} \approx \frac{\tau^{2 - \epsilon'}}{K^4}. \]
Now, for  $b \in A'$ we can write $b = a\lambda $ for some $\lambda \in S_\tau$. This implies 
\[ | (AA) \cap (AAa^{-1})b | = | (AA) \cap (AA) \lambda  | \geq |A_\lambda A_\lambda| \gg_{\epsilon'} \frac{\tau^{2 - \epsilon'}}{K^4}. \]
Since this holds for any $b \in A'$, we get
\[ d_*(A') \ll_{\epsilon'} \frac{|AA|^4}{|A'|(\tau^{2 - \epsilon'}/K^4)^3} = \frac{|AA|^4 K^{12}}{|A'| \tau^{6 - 3\epsilon'}}, \]
by letting $Q = AA$, $R = AAa^{-1}$, and $t = \tau^{2 - \epsilon'}/K^4$ in the definition of $d_*(A')$.

Combining this with Theorem~\ref{th:smallk} gives
\[ |A+A|^9 \gg_{\epsilon'} \frac{|A'|^{14 - \epsilon'}}{d_*(A')^{5}} \gg_{\epsilon'} \frac{|A'|^{19 - \epsilon'} \tau^{30 - 15\epsilon'}}{K^{60} |AA|^{20}}. \]
We then apply Equations~\eqref{eq:energybound},~\eqref{eq:taubound},~and~\eqref{eq:multiplicativecs}, in the specified order, to get 
\begin{align*}
K^{60} |AA|^{20} |A+A|^9 &\gg_{\epsilon'} |A'|^{19 - \epsilon'} \tau^{30 - 15\epsilon'} \geq \left(\frac{|S_\tau| \tau}{|A|} \right)^{19 - \epsilon'} \cdot \tau^{30 - 15\epsilon'} \\
& = \left(|S_\tau| \tau^{2}\right)^{19 - \epsilon'} \cdot \frac{ \tau^{11 - 14\epsilon'
}}{|A|^{19 - \epsilon'}}  \gtrsim E^*(A)^{19 - \epsilon'}  \cdot \frac{\tau^{11  - 14\epsilon'
}}{|A|^{19 - \epsilon'}} \\
& = \left(\frac{E^*(A)}{|A|} \right)^{19 - \epsilon'}\cdot \tau^{11 - 14\epsilon'} \gg \left(\frac{E^*(A)}{|A|} \right)^{19 - \epsilon'} \cdot  \left(\frac{E^*(A)}{|A|^2}\right)^{11 - 14\epsilon'}\\
& = \frac{E^*(A)^{30- 15\epsilon'} }{|A|^{41 -29\epsilon'}} \geq \frac{1}{|A|^{41 -29\epsilon'}} \cdot  \left(\frac{|A|^4}{|AA|}\right)^{30- 15\epsilon'}\\
& = \frac{|A|^{79 - 31\epsilon'}}{|AA|^{30- 15\epsilon'}}.
\end{align*}
That is, we have
\[ |AA|^{50} |A+A|^{9} \gtrsim_\epsilon \frac{|A|^{79 - \epsilon}}{K^{60}}.  \]

%By choosing $K$ sufficiently small depending on $|A|$, this implies that $\max\{|AA|, |A+A|\} \geq |A|^{79/59 - \delta} > |A|^{4/3 }$, where $\delta$ is a small constant depending on $K$.

When $K \leq |A|^{\delta}$ for some fixed $\delta > 0$, this implies $|A + A| + |AA| > |A|^{4/3+c}$ where $c > 0$ is an absolute constant.

%%%%%%%%%%%%%%%%%%%%%%%%%%%%%%%%%%%%%%%%%%%%%%%%%%%%%
%% Section: Large k
%%%%%%%%%%%%%%%%%%%%%%%%%%%%%%%%%%%%%%%%%%%%%%%%%%%%%

\subsection{The case when $K$ is large}
\label{sec:large-k}

Now we handle the case that $K > |A|^\delta$, where $\delta$ is the parameter chosen at the end of the previous section.
For convenience of notation, we label the slopes in $S_\tau$ by distinct integers, and let $A_i = A_{\lambda_i} = A \cap A{\lambda_i}$ for each $\lambda_i \in S_\tau$.

A key observation, first exploited by Solymosi in the real case \cite{solymosi08} and by Solymosi and Wong \cite{sw17} for quaternions, is that, for distinct $\lambda_i, \lambda_j \in S_\tau$ and $a_1,a_2 \in A_i, b_1,b_2 \in A_j$ with $(a_1,b_1) \neq (a_2,b_2)$, we have $(a_1,  a_1\lambda_i) + (b_1, b_1 \lambda_j) \neq (a_2, a_2 \lambda_i) + (b_2, b_2 \lambda_j)$.
Indeed, since $\lambda_i \neq \lambda_j$, we have that $(\lambda_i - \lambda_j)$ is invertible.
Hence, given $x,y \in \mathbb{H}$, we can solve the equations
\[a + b =  x, \qquad  a \lambda_i +  b \lambda_j = y   \]
uniquely for $a,b$.

For pairs of distinct elements $\{\lambda_i, \lambda_j\}, \{\lambda_k, \lambda_\ell\} \in \binom{S_\tau}{2}$, we say that $\{\lambda_i, \lambda_j\}$ {\em conflicts} with $\{\lambda_k, \lambda_\ell\}$ if $\{\lambda_i, \lambda_j\} \neq \{\lambda_k, \lambda_\ell\}$ and 
$$((A_{i} + A_{j}) \times ( A_{i}\lambda_i + A_{j} \lambda_j)) \cap ((A_{k} + A_{l}) \times ( A_{k} \lambda_k +  A_{l} \lambda_l)) \neq \emptyset.$$

In order to find a large set of distinct elements of $(A+A)\times(A+A)$, we will consider sums $((a, a\lambda_i) + (b, b\lambda_j)) \in ((A_i,  A_i \lambda_i) + (A_j,  A_j \lambda_j))$ for edges $\{\lambda_i, \lambda_j\}$ in a carefully selected graph $G \subset \binom{S_\tau}{2}$.
We construct $G$ so that few pairs of edges conflict.
For each pair of edges that does conflict, we apply an observation of Konyagin and Shkredov \cite{ks15} (adapted from the real to quaternionic case) to bound the number of elements of $(A+A) \times (A+A)$ that we count more than once.

First we give the observation of Konyagin and Shkredov.
Using the assumption on the additive energy of each $A_i$, we can bound the number of repeated sums coming from each pair of conflicting edges.
\begin{lemma}\label{lem:intersectionSize}
	Let $\{\lambda_i, \lambda_j\}$ , $\{\lambda_k, \lambda_\ell\}$ be a pair of conflicting edges.
	Then,
	\begin{equation}\label{eq:conflictIntersection}
	|((A_{i} + A_{j}) \times ( A_{i} \lambda_i + A_{j} \lambda_j)) \cap ((A_{k} + A_{\ell}) \times ( A_{k}\lambda_k + A_{\ell}  \lambda_\ell ))| \lesssim \tau^2 K^{-1/2}.\end{equation}
\end{lemma}

\begin{proof}
	Suppose, without loss of generality, that $\lambda_\ell$ is distinct from $\lambda_i$ and $\lambda_j$.
	Note that quadruples $(a_i, a_j, a_k, a_\ell) \in A_i \times A_j \times A_k \times A_\ell$ that contribute to (\ref{eq:conflictIntersection}) satisfy
	\begin{align}
	\label{eq:equalXCoord}a_i + a_j &= a_k + a_\ell, \\
	\label{eq:equalYCoord} a_i \lambda_i + a_j \lambda_j &= a_k \lambda_k +  a_\ell \lambda_\ell.
	\end{align}
	
	Subtracting Equation $\eqref{eq:equalXCoord}$ multiplied by $\lambda_\ell$ on the right from Equation $\eqref{eq:equalYCoord}$, we obtain
	\begin{equation}\label{eq:combinedSum}
	 a_i(\lambda_i - \lambda_\ell) + a_j(\lambda_j - \lambda_\ell) = a_k (\lambda_k - \lambda_\ell).
	\end{equation}
	
	Hence, $|((A_{i} + A_{j}) \times ( A_{i} \lambda_i + A_{j} \lambda_j)) \cap ((A_{k} + A_{\ell}) \times ( A_{k}\lambda_k + A_{\ell}  \lambda_\ell ))|$ is bounded above by the number $T$ of $(a_i,a_j,a_k) \in A_i \times A_j \times A_k$ that satisfy Equation \eqref{eq:combinedSum}.
	
	Applying Cauchy-Schwarz twice, we obtain
	\begin{align*}
	T &\leq |A_{k}|^{1/2} E^+(A_{i}(\lambda_i - \lambda_l), A_{j}(\lambda_j-\lambda_l))^{1/2},\\
	&\leq |A_{k}|^{1/2} E^+(A_{i})^{1/4} E^+(A_{j})^{1/4}.
	\end{align*}
	By construction, $|A_i| \approx \tau$ and $E^+(A_i) \approx \tau^3 K^{-1}$ for each $i$.
	Hence,
	\[T \lesssim \tau^2 K^{-1/2},\]
	as claimed.
\end{proof}
We also need to limit the number of pairs $(\lambda_i,\lambda_j),(\lambda_k, \lambda_\ell) \in G$ for which this intersection is non-empty.
This involves controlling when such intersections can occur, and choosing $G$ accordingly.

First, we establish that any ratio $(a_i + a_j)^{-1}(a_i \lambda_i + a_j \lambda_j)$ with $a_i \in A_i, a_j \in A_j$ must be close to $\lambda_i$.
This is Lemma 7 in Solymosi and Wong \cite{sw17}.

\begin{lemma}\label{lem:conflictBall}
	For any distinct $\lambda_i, \lambda_j \in S_\tau$ and $a_i \in A_i$, $a_j \in A_j$, we have
	$$\|(a_i + a_j)^{-1}(a_i \lambda_i + a_j \lambda_j) - \lambda_i\| \leq \|\lambda_j - \lambda_i\|.$$
\end{lemma}

\begin{proof}
	Recall that $a_i$ and $a_j$ are in the same orthant, so $\|a_i+a_j\| \geq \|a_j\|$.
	Using the identity $(a_i+a_j)^{-1} = (1- (a_i+a_j)^{-1}a_j)a_i^{-1}$, we have
	\begin{align*}
	\|(a_i+a_j)^{-1}(a_i \lambda_i + a_j \lambda_j) - \lambda_i\| &= \| (1-(a_i+a_j)^{-1}a_j)\lambda_i +  (a_i+a_j)^{-1}a_j \lambda_j - \lambda_i \|,\\
	&= \|(a_i + a_j)^{-1} a_j (\lambda_j - \lambda_i) \|, \\
	&= \|\lambda_j - \lambda_i\|\,\|a_j\|\,\|a_i+a_j\|^{-1}, \\
	&\leq \| \lambda_j-\lambda_i \|.
	\end{align*}
\end{proof}

Now we define $G$.
Let $M$ be a positive integer parameter that we will fix later.
Let $G$ be the graph on vertex set $S_\tau$ formed by joining each $\lambda \in S_\tau$ to its $M$ closest neighbors, breaking ties arbitrarily.
Since the degree of each vertex of $G$ is at least $M$, the number of edges in $G$ is at least $M |S_\tau|/2$.
We also need to bound the number of conflicting pairs of edges in $G$.

Let $B_i$ be the smallest closed ball centered at $\lambda_i$ that contains at least $M+1$ points of $S_\tau$ (including $\lambda_i$ itself), and let $R_i$ be the radius of $B_i$.
While a given ball $B_i$ may contain an arbitrary number of points of $S_\tau$, the interior of each ball contains at most $M$ points.

Note that, if $\{\lambda_i, \lambda_j\} \in G$, then at least one of $\lambda_i \in B_j$ or $\lambda_j \in B_i$.
In addition, from Lemma~\ref{lem:conflictBall}, we have that, for any distinct $\lambda_i, \lambda_k$, if there exist $\lambda_j \in B_i, \lambda_\ell \in B_k$ such that $\{\lambda_i,\lambda_j\}$ conflicts with $\{\lambda_k, \lambda_\ell\}$, then $B_i \cap B_k \neq \emptyset$.

For each $\lambda_i$, let $L_i$ be the set of $\lambda_j$ such that $R_j \geq R_i$ and $B_i \cap B_j \neq \emptyset$.
Given any $\lambda_i$, and $\lambda_j \in L_i$, the number of pairs of edges $\{\lambda_i, \lambda_k\}, \{\lambda_j, \lambda_\ell\}$ that conflict is bounded above by the product of the degrees of $\lambda_i$ and $\lambda_j$.
Hence, to bound the number of conflicting pairs of edges, it will suffice to place upper bounds on the degree of each vertex of $G$, and the size of each set $L_i$.
This is done in Lemmas \ref{lem:degreeBound} and \ref{lem:conflictBound}, respectively.

The proofs of both Lemma \ref{lem:degreeBound} and \ref{lem:conflictBound} both rely on similar geometric averaging arguments.
In each case, for a fixed $\lambda_i$, we bound the number of interesting $\lambda_j$ in an arbitrary cone of constant size with vertex $\lambda_i$.
By linearity of expectation, this gives an upper bound on the number of interesting $\lambda_j$ in all directions from $\lambda_i$.

The following simple geometric lemma used in both proofs.

\begin{lemma}\label{lem:vectorDiff}
	Let $\mathbf{x}, \mathbf{y}$ be arbitrary vectors in $\mathbb{R}^4$, and let $\mathbf{u}$ be a unit vector.
	Let $\mathbf{x}^\perp$ and $\mathbf{y}^\perp$ be the projections of $\mathbf{x}$ and $\mathbf{y}$ onto the subspace orthogonal to $\mathbf{u}$.
	Then,
	$$\| \mathbf{x} - \mathbf{y} \|^2 \leq \|\mathbf{x}\|^2 + \|\mathbf{y}\|^2 - 2(\mathbf{x}\cdot \mathbf{u})(\mathbf{y}\cdot \mathbf{u})+2\|\mathbf{x}^\perp\|\|\mathbf{y}^\perp\|.$$
\end{lemma}
\begin{proof}
	\begin{align*}
	\|\mathbf{x} - \mathbf{y}\|^2
	&= \|(\mathbf{x}\cdot \mathbf{u} - \mathbf{y} \cdot \mathbf{u})\mathbf{u} + \mathbf{x}^\perp - \mathbf{y}^\perp \|^2, \\
	&= (\mathbf{x}\cdot \mathbf{u} - \mathbf{y} \cdot \mathbf{u})^2 + \|\mathbf{x}^\perp - \mathbf{y}^\perp\|^2, \\
	&\leq (\mathbf{x}\cdot \mathbf{u} - \mathbf{y} \cdot \mathbf{u})^2 + (\|\mathbf{x}^\perp\| + \|\mathbf{y}^\perp\|)^2, \\
	&= \|\mathbf{x}\|^2 + \|\mathbf{y}\|^2 - 2(\mathbf{x}\cdot \mathbf{u})(\mathbf{y}\cdot \mathbf{u})+2\|\mathbf{x}^\perp\|\|\mathbf{y}^\perp\|.
	\end{align*}
\end{proof}

\begin{lemma}\label{lem:degreeBound}
	For each $i$, the degree $\deg(\lambda_i)$ of $\lambda_i$ in $G$ is bounded by
	\[\deg(\lambda_i) \ll M.\]
\end{lemma}
\begin{proof}
	Take $\lambda_i$ to be at the origin.
	Let $\mathbf{u}$ be an arbitrary unit vector, and let $C$ be the cone
	$$C = \{\mathbf{x} : \mathbf{x} \cdot \mathbf{u} > (\sqrt{3}/2)\|\mathbf{x}\|\}.$$
	
	We will use the decomposition $\mathbf{x} = (\mathbf{x} \cdot \mathbf{u})\mathbf{u} + \mathbf{x}^\perp$.
	Note that, if $\mathbf{x} \in C$, then
	$$\|\mathbf{x}^\perp\|^2 = \|\mathbf{x}\|^2 - (\mathbf{x} \cdot \mathbf{u})^2 < (1/4)\|\mathbf{x}\|^2.$$
	
	Let $\mathbf{x}, \mathbf{y}$ be vectors in $C$ with $\|\mathbf{x}\|\geq\|\mathbf{y}\|$.
	By Lemma \ref{lem:vectorDiff},
	\begin{align*}
	\| \mathbf{x} - \mathbf{y} \|^2 &\leq \|\mathbf{x}\|^2 + \|\mathbf{y}\|^2 - 2(\mathbf{x}\cdot \mathbf{u})(\mathbf{y}\cdot \mathbf{u})+2\|\mathbf{x}^\perp\|\|\mathbf{y}^\perp\|, \\
	&< \|\mathbf{x}\|^2 + \|\mathbf{y}\|^2 - \|\mathbf{x}\|\|\mathbf{y}\|,\\
	&\leq \|\mathbf{x}\|^2.
	\end{align*}
	
	To put this another way, the distance between any pair of points in $C$ is less than the distance from the more distant point to $\lambda_i$.
	
	Consequently, $\lambda_i$ has at most $M$ neighbors in $C$.
	Indeed, suppose for contradiction that there are $M+1$ neighbors in $C$, and let $\lambda$ be a neighbor at the largest distance from $\lambda_i$.
	Then, all of the remaining $M$ of the neighbors are closer to $\lambda$ than $\lambda_i$ is.
	This is a contradiction, since only the $M$ closest points to $\lambda$ are selected as neighbors.
	
	The preceding argument applies for an arbitrary unit vector $\mathbf{u}$.
	Suppose, for contradiction, that $\lambda_i$ has more than $cM$ neighbors for a sufficiently large constant $c$.
	Then, the expected number of neighbors in the cone corresponding to a uniformly random $\mathbf{u}$ will be greater than $M$, but we have shown that there cannot be any $\mathbf{u}$ for which this number is greater than $M$.
\end{proof}

Next we bound $|L_i|$ for an arbitrary $\lambda_i$.

\begin{lemma}\label{lem:conflictBound}
	For each $i$,
	$$|L_i| \ll M.$$
\end{lemma}
\begin{proof}
	Take $\lambda_i$ to be at the origin.
	Let $\mathbf{u}$ be an arbitrary unit vector, and let $C$ be the cone
	$$C = \{\mathbf{x} : \mathbf{x} \cdot \mathbf{u} > 0.99 \|\mathbf{x}\|\}.$$
	Note that, if $\mathbf{x} \in C$, then
	$$\|\mathbf{x}^\perp\|^2 = \|\mathbf{x}\|^2 - (\mathbf{x} \cdot \mathbf{u})^2 < 0.02 \|\mathbf{x}\|^2.$$
	If $\mathbf{x}, \mathbf{y}$ are vectors in $C$, then by Lemma \ref{lem:vectorDiff},
	\begin{align}
	\| \mathbf{x} - \mathbf{y} \|^2 &\leq \|\mathbf{x}\|^2 + \|\mathbf{y}\|^2 - 2(\mathbf{x}\cdot \mathbf{u})(\mathbf{y}\cdot \mathbf{u})+2\|\mathbf{x}^\perp\|\|\mathbf{y}^\perp\|, \nonumber \\
	&< \|\mathbf{x}\|^2 + \|\mathbf{y}\|^2 - 1.92\|\mathbf{x}\|\|\mathbf{y}\|. \label{eq:L_iDistance}
	\end{align}

	Define the following subsets of $C$:
	\begin{align*}
	C_1 &= \{\mathbf{x} \in C : \|\mathbf{x}\| < R_i\}, \\
	C_2 &= \{\mathbf{x} \in C : R_i \leq \|\mathbf{x}\| < 1.5 R_i \}, \\
	C_3 &= \{\mathbf{x} \in C : 1.5 R_i \leq \|\mathbf{x}\| \}.
	\end{align*}
	We will show that each of $C_1, C_2,$ and $C_3$ contains at most $M$ points of $L_i$.
	
	Since every point of $S_\tau$ that is at distance less than $R_i$ from $\lambda_i$ is a neighbor of $\lambda_i$, it is immediate that $C_1$ contains fewer than $M$ points of $S_\tau$.
	
	Let $\mathbf{x}, \mathbf{y}$ be vectors in $C_2$ with $R_i \leq \|\mathbf{y}\| \leq \|\mathbf{x}\| < 1.5 R_i$.
	By Equation~\eqref{eq:L_iDistance},
	\begin{align*}
	\| \mathbf{x} - \mathbf{y} \|^2 &< \|\mathbf{x}\|^2 + \|\mathbf{y}\|^2 - 1.92\|\mathbf{x}\|\|\mathbf{y}\|, \\
	&\leq \|\mathbf{x}\|^2 + \|\mathbf{x}\|\|\mathbf{y}\| - 1.92\|\mathbf{x}\|\|\mathbf{y}\|, \\
	&\leq \|\mathbf{x}\|^2 - 0.92 \|\mathbf{x}\| R_i, \\
	&< 0.87 R_i^2.
	\end{align*}
	
	In other words, each pair of vectors in $C_2$ is at distance less than $R_i$.
	By assumption, if $R$ is the radius of the ball associated to any $\lambda \in L_i$, we have $R \geq R_i$.
	If there are $M+1$ points of $L_i$ in $C_2$, then the ball associated to each of these points contains all $M+1$ of the points in its interior, which contradicts the construction of the balls.
	
	Let $\mathbf{x}, \mathbf{y}$ be vectors in $C_3$ with $1.5 R_i \leq \|\mathbf{y}\| \leq \|\mathbf{x}\|$.
	We claim that
	\begin{equation}\label{eq:C3claim}
	\|\mathbf{x} - \mathbf{y} \|^2 < (\|\mathbf{x}\| - R_i)^2.
	\end{equation}
	From (\ref{eq:L_iDistance}), in order to prove (\ref{eq:C3claim}), it is enough to show that
	\begin{align}
	\|\mathbf{y}\|^2 + 2\|\mathbf{x}\|R_i - 1.92 \|\mathbf{x}\| \|\mathbf{y}\| - R_i^2 < 0,& \label{eq:C3functToMaximize}\\
	\text{under the assumption that} \hspace{1cm} \|\mathbf{x}\| \geq \|\mathbf{y}\| \geq 1.5 R_i.& \label{eq:C3constraint}
	\end{align}
	
	When the left hand side of (\ref{eq:C3functToMaximize}) is constrained to the line $\|\mathbf{x}\|=\|\mathbf{y}\|$, we have
	$$-0.98\|\mathbf{x}\|^2 + 2\|\mathbf{x}\|R_i + R_i^2.$$
	The derivative of this is negative for $\|\mathbf{x}\| \geq 1.5 R_i$.
	Furthermore, when $\|\mathbf{y}\| \geq 1.5R_i$, the left hand side of (\ref{eq:C3functToMaximize}) is decreasing in $\|\mathbf{x}\|$.
	Hence, it suffices to consider the point $\|\mathbf{x}\| = \|\mathbf{y}\| = 1.5 R_i$.
	At this point, the left hand side of (\ref{eq:C3functToMaximize}) is $-.07 R_i < 0$, as claimed.

%	Considering the partial derivatives of the left hand side of (\ref{eq:C3functToMaximize}), we can see that this quadric attains its unique unconstrained maximum at $\|\mathbf{y}\| = (2/1.92)$, $\|\mathbf{x}\| = (2/1.92)^2$.
%	Since this point is outside of the region defined by (\ref{eq:C3constraint}), it is enough to consider the boundaries of this region.

%	When (\ref{eq:C3functToMaximize}) is constrained to the line $\|\mathbf{y}\|=1.5R$, we have
%	$$1.25R_i^2 - 0.88\|\mathbf{x}\|R_i.$$
%	Here, the derivative is always negative.

	The meaning of (\ref{eq:C3claim}) is that, given any two points in $C_3$, the distance from the point that is further from $\lambda_i$ to the nearest point of $B_i$ is greater than the distance between the two points.
	Consequently, if $C_3$ contains $M+1$ points of $L_i$, the ball around the furthest point (which intersects $B_i$) must contain all $M+1$ points strictly in its interior, which is a contradiction.
\end{proof}

Let $T \subset G^2$ be the set of pairs of edges in $G$ that conflict.
Applying Lemmas \ref{lem:degreeBound} and \ref{lem:conflictBound},
\begin{equation}\label{eq:conflictingPairsBound}
|T| \leq \sum_i \deg(\lambda_i) \sum_{\lambda_j \in L_i} \deg(\lambda_j) \ll |S_\tau| M^3.\end{equation}

We have
\[(A+A) \times (A+A) \supset \bigcup_{\{\lambda_i, \lambda_j\} \in G} (A_i + A_j)\times ( A_i \lambda_i +  A_j \lambda_j).\]
Using inclusion-exclusion and Lemma \ref{lem:intersectionSize}, this implies that
\[|(A+A) \times (A+A)| \geq |G|\tau^2 - c|T|\tau^2 K^{-1}\]
for some $c>0$.
Since the degree of each $\lambda_i \in G$ is at least $M$, we have $|G| \geq |S_\tau|M/2$.
Together with Equation \eqref{eq:conflictingPairsBound}, we get
\[|(A+A)\times (A+A)| \geq |S_\tau|M \tau^2/2 - c |S_\tau| M^3 \tau^2 K^{-1/2}.\]
Taking $M = c'K^{1/6}$ for a sufficiently small constant $c'$, we obtain
\[|A+A|^2 \gg |S_\tau| \tau^2 K^{1/6} \gtrsim E^*(A) K^{1/6} \geq |A|^4 |AA|^{-1} K^{1/6} .\]
Assuming that $K > |A|^{\delta}$ for some fixed $\delta > 0$, this shows that $|A+A|+|AA| \gtrsim |A|^{4/3 + \delta/18}$, as claimed.

\section*{Acknowledgments}
We would like to thank Oliver Roche-Newton for sharing his unpublished notes on Konyagin and Shkredov's sum-product bound.
We would like to thank Adam Sheffer for extensive discussions on an earlier version of this paper.

Work on this project by Ben Lund was supported by NSF grant DMS-1344994 (RTG in Algebra,
Algebraic Geometry, and Number Theory, at the University of Georgia).
\bibliographystyle{plain}
\bibliography{sum-product-ref}

\begin{thebibliography}{10}

\bibitem{chang05}
M.-C. Chang.
\newblock A sum-product estimate in algebraic division algebras.
\newblock {\em Israel Journal of Mathematics}, 150(1):369--380, 2005.

\bibitem{elekes97}
G.~Elekes.
\newblock On the number of sums and products.
\newblock {\em Acta Arithmetica}, 81(4):365--367, 1997.

\bibitem{er03}
G.~Elekes and I.~Ruzsa.
\newblock Few sums, many products.
\newblock {\em Studia Scientiarum Mathematicarum Hungarica}, 40(3):301--308,
  2003.

\bibitem{es83}
P.~Erd{\H{o}}s and E.~Szemer{\'e}di.
\newblock On sums and products of integers.
\newblock In {\em Studies in pure mathematics}, pages 213--218. Springer, 1983.

\bibitem{ford98}
K.~Ford.
\newblock Sums and products from a finite set of real numbers.
\newblock {\em The Ramanujan Journal}, 2(1-2):59--66, 1998.

\bibitem{gs16}
A.~Granville and J.~Solymosi.
\newblock Sum-product formulae.
\newblock In {\em Recent Trends in Combinatorics}, pages 419--451. Springer,
  2016.

\bibitem{kr13}
S.~Konyagin and M.~Rudnev.
\newblock On new sum-product--type estimates.
\newblock {\em SIAM Journal on Discrete Mathematics}, 27(2):973--990, 2013.

\bibitem{ks15}
S.~Konyagin and I.~Shkredov.
\newblock On sum sets of sets having small product set.
\newblock {\em Proceedings of the Steklov Institute of Mathematics},
  290(1):288--299, 2015.

\bibitem{ks16}
S.~Konyagin and I.~Shkredov.
\newblock New results on sums and products in $\mathbb{R}$.
\newblock {\em Proceedings of the Steklov Institute of Mathematics},
  294(1):78--88, 2016.

\bibitem{li11}
L.~Li.
\newblock On a theorem of {S}choen and {S}hkredov on sumsets of convex sets.
\newblock {\em arXiv preprint arXiv:1108.4382}, 2011.

\bibitem{nathanson97}
M.~Nathanson.
\newblock On sums and products of integers.
\newblock {\em Proceedings of the American Mathematical Society}, 125(1):9--16,
  1997.

\bibitem{rss16}
M.~Rudnev, I.~Shkredov, and S.~Stevens.
\newblock On the energy variant of the sum-product conjecture.
\newblock {\em arXiv preprint arXiv:1607.05053}, 2016.

\bibitem{shakan18}
George Shakan.
\newblock On higher energy decompositions and the sum--product phenomenon.
\newblock In {\em Mathematical Proceedings of the Cambridge Philosophical
  Society}, pages 1--19. Cambridge University Press, 2018.

\bibitem{solymosi08}
J.~Solymosi.
\newblock Bounding multiplicative energy by the sumset.
\newblock {\em Advances in Mathematics}, 222(2):402--408, 2009.

\bibitem{st12}
J.~Solymosi and T.~Tao.
\newblock An incidence theorem in higher dimensions.
\newblock {\em Discrete \& Computational Geometry}, 48(2):255--280, 2012.

\bibitem{sw17}
J.~Solymosi and C.~Wong.
\newblock An application of kissing numbers in sum-product estimates.
\newblock {\em arXiv preprint arXiv:1709.08758}, 2017.

\end{thebibliography}

\end{document}